\documentclass[11pt,a4paper]{article}
\usepackage[T1]{fontenc}
\usepackage[utf8]{inputenc}
\usepackage{lmodern}
\usepackage{amsmath,amssymb,amsthm}
\usepackage[a4paper,margin=28mm]{geometry}
\usepackage{microtype}
\usepackage[hidelinks]{hyperref}
\usepackage{cite}
\numberwithin{equation}{section}
\newtheorem{theorem}{Theorem}[section]
\newtheorem{corollary}[theorem]{Corollary}
\newtheorem{lemma}[theorem]{Lemma}
\theoremstyle{remark}

\newcommand{\N}{\mathbb N}
\newcommand{\Sset}{\mathcal S}
\title{\Large\bfseries A necessary condition for a Romanoff type theorem}
\author{Xiao-Hui Yan\thanks{Corresponding author.}\hspace{0.5em} and Ya-Ting Xu\\[5pt]
\small School of Mathematics and Statistics,
Anhui Normal University,\\
\small Wuhu 241002, China\\[3pt]
\small Email: \texttt{yanxiaohui\_1992@163.com};
\texttt{980181682@qq.com}}
\date{}

\begin{document}
\maketitle
\begin{abstract}
Let $\{a_k\}_{k\ge1}$ be an infinite strictly increasing sequence
of positive integers with $a_k\mid a_{k+1}$ for every integer $k\ge1$.
Dai and Pan proved that the set of integers of the form
$p^2+q^2+a_k+a_l$, where $p,q$ are primes and $k,l\ge1$, has positive
lower asymptotic density if $\limsup_{k\to\infty}\frac{\log a_k}{k}<\infty$ and
$\sup_{k\ge1}\sum_{d\mid a_k}\frac1d<\infty$. They also asked whether the latter
condition is necessary. In this paper, we show that this condition
is necessary by proving that if $\sup_{k\ge1}\sum_{d\mid a_k}\frac1d=\infty$,
then this set has asymptotic density zero.
\end{abstract}

\section{Introduction}

Let $\N$ be the set of all positive integers. For a set
$\mathcal A\subseteq\N$, write $\mathcal A(x)=|\mathcal A\cap[1,x]|$.
Its lower and upper asymptotic densities are defined by
\[
\underline d(\mathcal A)=\liminf_{x\to\infty}\frac{\mathcal A(x)}x,
\qquad
\overline d(\mathcal A)=\limsup_{x\to\infty}\frac{\mathcal A(x)}x,
\]
respectively. If $\underline d(\mathcal A)=\overline d(\mathcal A)$,
then its asymptotic density exists and is denoted by $d(\mathcal A)$.

In 1934, Romanoff \cite{Romanoff} proved that the set
\[
\{p+2^k:p\text{ is prime and }k\ge0\}
\]
has positive lower asymptotic density. This result has been generalized in various directions;
see \cite{BallotLuca,ChenSparse,ChenSun,ChenXu,DaiPanPolynomial,
DingVerwee,DingZhou,Dubickas,Kuan,Lee,LiPan,LiuXue,
ShparlinskiWeingartner,SunChen,YangChen}.

In 1950, Erd\H{o}s
\cite[Theorem 4]{Erdos} extended Romanoff's theorem to strictly increasing sequences
$\{a_k\}_{k\ge1}$ of positive integers satisfying
$a_k\mid a_{k+1}$. He proved that the set
$\{p+a_k:p\text{ is prime and }k\ge1\}$ has positive lower
asymptotic density if and only if
\begin{equation}\label{eq:growth}
\limsup_{k\to\infty}\frac{\log a_k}{k}<\infty
\end{equation}
and
\begin{equation}\label{eq:divisors}
\sup_{k\ge1}\sum_{d\mid a_k}\frac1d<\infty.
\end{equation}

Liu, Liu and Zhan \cite{LLZ} considered an analogue of Romanoff's theorem for sums of two prime squares and proved that
\[
\{p^2+q^2+2^k+2^l:p,q\text{ are primes and }k,l\ge0\}
\]
has positive lower asymptotic density. Recently, Dai and Pan
\cite[Theorem 1.1]{DaiPan} extended the result of Liu, Liu and Zhan
to strictly increasing sequences $\{a_k\}_{k\ge1}$ of positive
integers with $a_k\mid a_{k+1}$ satisfying \eqref{eq:growth} and
\eqref{eq:divisors}. They proved the following result.
\begin{theorem}[\cite{DaiPan}, Theorem 1.1]\label{thm:daipan}
Let $\{a_k\}_{k\ge1}$ be an infinite strictly increasing sequence
of positive integers such that $a_k\mid a_{k+1}$ for every integer
$k\ge1$. If $\{a_k\}_{k\ge1}$ satisfies \eqref{eq:growth} and
\eqref{eq:divisors}, then the set
\[
\{p^2+q^2+a_k+a_l:
p,q\text{ are primes and }k,l\ge1\}
\]
has positive lower asymptotic density.
\end{theorem}

In \cite[Remark]{DaiPan}, they noted that \eqref{eq:growth} is necessary and asked whether
\eqref{eq:divisors} is necessary. In this paper, we answer this question by proving the following
result.

\begin{theorem}\label{thm:zero}
Let $\{a_k\}_{k\ge1}$ be an infinite strictly increasing sequence of positive integers such that
$a_k\mid a_{k+1}$ for every integer $k\ge1$. If
\[
\sup_{k\ge1}\sum_{d\mid a_k}\frac1d=\infty,
\]
then the set
\[
\{p^2+q^2+a_k+a_l:
p,q\text{ are primes and }k,l\ge1\}
\]
has asymptotic density zero.
\end{theorem}

Combining Theorems~\ref{thm:daipan} and~\ref{thm:zero} with
the necessity of \eqref{eq:growth}, we obtain the following result.

\begin{corollary}\label{cor:criterion}
Let $\{a_k\}_{k\ge1}$ be an infinite strictly increasing sequence of positive integers such that
$a_k\mid a_{k+1}$ for every integer $k\ge1$. Then the set
\[
\{p^2+q^2+a_k+a_l:
p,q\text{ are primes and }k,l\ge1\}
\]
has positive lower asymptotic density if and only if
\eqref{eq:growth} and \eqref{eq:divisors} hold.
\end{corollary}

Throughout this paper, $p$ and $q$ denote primes, $\pi(x)$ denotes
the number of primes not exceeding $x$, and $\varphi$ is Euler's
totient function. For functions $f(x)$ and $g(x)$ with $g(x)>0$,
we write $f(x)=O(g(x))$ if there exists a constant $C>0$ such that
$|f(x)|\le Cg(x)$ for all sufficiently large $x$, and
$f(x)=o(g(x))$ if $f(x)/g(x)\to0$ as $x\to\infty$.

\section{Preliminary lemmas}

For an odd prime $q$, let
\[
\chi(q)=\begin{cases}
1,&q\equiv1\pmod4,\\
-1,&q\equiv3\pmod4.
\end{cases}
\]

\begin{lemma}\label{lem:local}
Let $\mathcal F$ be a nonempty finite set of odd primes, and let
\[
H=\sum_{q\in \mathcal F}\frac1q,\qquad Q=\prod_{q\in \mathcal F}q,
\qquad
T(n)=\sum_{\substack{q\in \mathcal F\\q\mid n}}\chi(q)
-\sum_{q\in \mathcal F}\frac{\chi(q)}q.
\]
Then
\begin{equation}\label{eq:integer-bound}
\left|\left\{1\le n\le Q:T(n)\ge\frac H2\right\}\right|
\le\frac{4Q}{H}
\end{equation}
and
\begin{equation}\label{eq:unit-bound}
\frac1{\varphi(Q)^2}
\left|\left\{(u,v):1\le u,v\le Q,\ (uv,Q)=1,
T(u^2+v^2)<\frac H2\right\}\right|
\le\frac{10}H.
\end{equation}
\end{lemma}

\begin{proof}
For any primes $p,q\in \mathcal F$ with $p\ne q$, we have
\[
|\{1\le n\le Q:p\mid n\}|=\frac Qp,
\qquad
|\{1\le n\le Q:pq\mid n\}|=\frac Q{pq}.
\]
It follows that
\begin{align*}
\frac1Q\sum_{n=1}^{Q}T(n)^2
&=\frac1Q\sum_{n=1}^{Q}
\left(
\sum_{\substack{q\in\mathcal F\\q\mid n}}\chi(q)
-\sum_{q\in\mathcal F}\frac{\chi(q)}q
\right)^2\\
&=\frac1Q\sum_{n=1}^{Q}
\left(
\sum_{\substack{q\in\mathcal F\\q\mid n}}\chi(q)\right)^2
-\left(\sum_{q\in\mathcal F}\frac{\chi(q)}q
\right)^2\\
&=\frac1Q\sum_{p,q\in\mathcal F}\chi(p)\chi(q)
\sum_{\substack{1\le n\le Q\\q\mid n,~p\mid n}}1
-\sum_{p,q\in\mathcal F}\frac{\chi(p)\chi(q)}{pq}\\
&=\sum_{q\in\mathcal F}\frac1q
+\sum_{\substack{p,q\in\mathcal F\\p\ne q}}
\frac{\chi(p)\chi(q)}{pq}
-\sum_{q\in\mathcal F}\frac{1}{q^2}
-\sum_{\substack{p,q\in\mathcal F\\p\ne q}}
\frac{\chi(p)\chi(q)}{pq}\\
&=\sum_{q\in\mathcal F}\frac1q
\left(1-\frac1q\right)
\le H,
\end{align*}
where we used $\chi(q)^2=1$ in the penultimate equality. Thus
\[
\frac{H^2}4
\left|\left\{1\le n\le Q:T(n)\ge\frac H2\right\}\right|
\le\sum_{n=1}^QT(n)^2\le QH,
\]
which proves \eqref{eq:integer-bound}.

We next consider the integer pairs $(u,v)$ with $1\le u,v\le Q$ and
$(uv,Q)=1$. For every prime $q\in \mathcal F$, since
\[
\left|\left\{t:1\le t\le q-1,
t^2\equiv-1\pmod q\right\}\right|=1+\chi(q),
\]
it follows that
\begin{align*}
&\left|\left\{(u,v):1\le u,v\le q-1,
u^2+v^2\equiv0\pmod q\right\}\right|\\
=&\left|\left\{(u,v):1\le u,v\le q-1,
(uv^{-1})^2\equiv-1\pmod q\right\}\right|\\
=&(q-1)(1+\chi(q)).
\end{align*}
By the Chinese remainder theorem, for any prime
$q\in\mathcal F$,
\begin{align}\label{eq:uvcount}
&\frac{1}{\varphi(Q)^2}
\left|\left\{
(u,v):1\le u,v\le Q,\ (uv,Q)=1,\ q\mid u^2+v^2
\right\}\right|\notag\\
=&
\frac{(q-1)(1+\chi(q))
\prod_{\substack{p\in\mathcal F\\p\ne q}}(p-1)^2
}{\varphi(Q)^2}
=\frac{1+\chi(q)}{q-1},
\end{align}
and for any distinct primes $p,q\in\mathcal F$
\begin{align}\label{eq:uvdoublecount}
&\frac1{\varphi(Q)^2}
\left|\left\{
(u,v):1\le u,v\le Q,\ (uv,Q)=1,\ pq\mid u^2+v^2
\right\}\right|\notag\\
=&
\frac{
(p-1)(1+\chi(p))(q-1)(1+\chi(q))
\prod_{\substack{\ell\in\mathcal F\\
\ell\ne p,q}}(\ell-1)^2
}{\varphi(Q)^2}\notag\\
=&
\frac{(1+\chi(p))(1+\chi(q))}{(p-1)(q-1)},
\end{align}
where we used $\varphi(Q)=\prod_{p\in\mathcal F}(p-1)$. Write
\[
m=\frac{1}{\varphi(Q)^2}
\sum_{\substack{1\le u,v\le Q\\(uv,Q)=1}}
T(u^2+v^2).
\]
Then
\begin{align}
m
&=\frac1{\varphi(Q)^2}
\sum_{\substack{1\le u,v\le Q\\(uv,Q)=1}}
\left(
\sum_{\substack{q\in\mathcal F\\q\mid u^2+v^2}}\chi(q)
-\sum_{q\in\mathcal F}\frac{\chi(q)}q
\right)\notag\\
&=\sum_{q\in\mathcal F}
\left(\frac{1+\chi(q)}{q-1}-\frac{\chi(q)}{q}\right)\label{eq:m}\\
&=\sum_{q\in\mathcal F}
\left(\frac1q+\frac{1+\chi(q)}{q(q-1)}\right)\notag\\
&=H+\sum_{q\in\mathcal F}
\frac{1+\chi(q)}{q(q-1)}
\ge H\label{eq:mH},
\end{align}
where we used \eqref{eq:uvcount} and $\chi(q)^2=1$ in the second equality and $1+\chi(q)\ge0$
for every $q\in\mathcal F$ in the last inequality. Thus, by \eqref{eq:m},
\begin{align*}
T(u^2+v^2)-m
&=\sum_{\substack{q\in\mathcal F\\q\mid u^2+v^2}}\chi(q)
-\sum_{q\in\mathcal F}\frac{\chi(q)}q
-\sum_{q\in\mathcal F}
\left(\frac{1+\chi(q)}{q-1}-\frac{\chi(q)}{q}\right)\\
&=\sum_{\substack{q\in\mathcal F\\q\mid u^2+v^2}}\chi(q)
-\sum_{q\in\mathcal F}\frac{1+\chi(q)}{q-1}.
\end{align*}
Hence,
\begin{align}
&\frac1{\varphi(Q)^2}
\sum_{\substack{1\le u,v\le Q\\(uv,Q)=1}}
\bigl(T(u^2+v^2)-m\bigr)^2\notag\\
=&
\frac1{\varphi(Q)^2}
\sum_{\substack{1\le u,v\le Q\\(uv,Q)=1}}
\left(
\sum_{\substack{q\in\mathcal F\\q\mid u^2+v^2}}
\chi(q)
\right)^2
-\left(
\sum_{q\in\mathcal F}\frac{1+\chi(q)}{q-1}
\right)^2\notag\\
=&
\sum_{q\in\mathcal F}\frac{1+\chi(q)}{q-1}
+\sum_{\substack{p,q\in\mathcal F\\p\ne q}}
\frac{(1+\chi(p))(1+\chi(q))}{(p-1)(q-1)}
-\left(
\sum_{q\in\mathcal F}\frac{1+\chi(q)}{q-1}
\right)^2\notag\\
=&
\sum_{q\in\mathcal F}
\left(
\frac{1+\chi(q)}{q-1}
-\frac{(1+\chi(q))^2}{(q-1)^2}
\right)\notag\\
\le&
\sum_{q\in\mathcal F}\frac{1+\chi(q)}{q-1}
\le\frac52H\label{eq:52H},
\end{align}
where the first equality follows from \eqref{eq:uvcount}, the second equality follows from \eqref{eq:uvcount}--\eqref{eq:uvdoublecount}, and the last inequality follows from $1+\chi(3)=0$ and
\[
\frac{1+\chi(q)}{q-1}
\le\frac2{q-1}
\le\frac5{2q}
\qquad(q\ge5).
\]
Noting that $H\le m$ in \eqref{eq:mH}, if $T(u^2+v^2)<H/2$, then
\[
\bigl|T(u^2+v^2)-m\bigr|
=m-T(u^2+v^2)
>m-\frac H2
\ge\frac H2.
\]
It follows from \eqref{eq:52H} that
\begin{align*}
&\frac1{\varphi(Q)^2}
\left|\left\{
(u,v):1\le u,v\le Q,\ (uv,Q)=1,\
T(u^2+v^2)<\frac H2
\right\}\right|\\
\le&\frac1{\varphi(Q)^2}
\left|\left\{
(u,v):1\le u,v\le Q,\ (uv,Q)=1,\
\bigl|T(u^2+v^2)-m\bigr|>\frac H2
\right\}\right|\\
\le&
\frac4{H^2\varphi(Q)^2}
\sum_{\substack{1\le u,v\le Q\\(uv,Q)=1}}
\bigl(T(u^2+v^2)-m\bigr)^2\\
\le&\frac4{H^2}\cdot\frac52H
=\frac{10}{H},
\end{align*}
which proves \eqref{eq:unit-bound}. This completes the proof of Lemma~\ref{lem:local}.
\end{proof}

\begin{lemma}\label{lem:primes}
Let $\mathcal F$ be a fixed nonempty finite set of odd primes, and let $H,Q,T$ be as in Lemma~\ref{lem:local}.
As $X\to\infty$, we have
\begin{equation}\label{eq:ordinary-count}
\left|\left\{1\le n\le X:T(n)\ge\frac H2\right\}\right|
\le\frac{4X}H+O_{\mathcal F}(1)
\end{equation}
and
\begin{equation}\label{eq:prime-count}
\left|\left\{(p,q):p,q\le\sqrt X,\
T(p^2+q^2)<\frac H2\right\}\right|
\le
\left(\frac{10}H+o_{\mathcal F}(1)\right)\pi(\sqrt X)^2.
\end{equation}
\end{lemma}

\begin{proof}
Clearly, $T(n)$ is periodic modulo $Q$. It follows from \eqref{eq:integer-bound} that
\begin{align*}
\left|\left\{1\le n\le X:T(n)\ge\frac H2\right\}\right|
&\le\frac{4X}{H}+O_{\mathcal F}(1),
\end{align*}
which proves \eqref{eq:ordinary-count}.

For any integer $u$ with $(u,Q)=1$, the prime number theorem for arithmetic
progressions gives
\[
|\{p\le\sqrt X:p\equiv u\pmod Q\}|
=\frac{\pi(\sqrt X)}{\varphi(Q)}+o_{\mathcal F}(\pi(\sqrt X))
\qquad(X\to\infty);
\]
see \cite[Chapter 20]{Davenport}. Then, for any integers $u,v$ with $(uv,Q)=1$,
\begin{align}
&\left|\left\{(p,q):p,q\le\sqrt X,\
p\equiv u\pmod Q,\ q\equiv v\pmod Q\right\}\right|\notag\\
=&
\left|\{p\le\sqrt X:p\equiv u\pmod Q\}\right|
\left|\{q\le\sqrt X:q\equiv v\pmod Q\}\right|\notag\\
=&
\left(\frac1{\varphi(Q)}+o_{\mathcal F}(1)\right)
\left(\frac1{\varphi(Q)}+o_{\mathcal F}(1)\right)
\pi(\sqrt X)^2\notag\\
=&
\left(\frac1{\varphi(Q)^2}+o_{\mathcal F}(1)\right)
\pi(\sqrt X)^2\label{eq:pi2}.
\end{align}

For any pair of primes $(p,q)$, if $(pq,Q)=1$, then there exist integers $1\le u,v<Q$ with $(uv,Q)=1$ such that $p\equiv u\pmod Q$ and $q\equiv v\pmod Q$, and so $p^2+q^2\equiv u^2+v^2\pmod Q$ and $T(p^2+q^2)=T(u^2+v^2)$. Hence,
\begin{align*}
&\left|\left\{(p,q):p,q\le\sqrt X,\ (pq,Q)=1,\
T(p^2+q^2)<\frac H2\right\}\right|\\
=&\sum_{\substack{1\le u,v\le Q\\
(uv,Q)=1\\T(u^2+v^2)<H/2}}
\left|\left\{(p,q):p,q\le\sqrt X,\
p\equiv u\pmod Q,\ q\equiv v\pmod Q
\right\}\right|\\
=&
\left(
\frac1{\varphi(Q)^2}
\sum_{\substack{1\le u,v\le Q\\(uv,Q)=1\\
T(u^2+v^2)<H/2}}1
+o_{\mathcal F}(1)
\right)\pi(\sqrt X)^2\\
\le&
\left(\frac{10}{H}+o_{\mathcal F}(1)\right)\pi(\sqrt X)^2,
\end{align*}
where the last equality follows from \eqref{eq:pi2}, and the last inequality follows from \eqref{eq:unit-bound}. Since
\begin{align*}
&\left|\left\{(p,q):p,q\le\sqrt X,\ (pq,Q)>1,\
T(p^2+q^2)<\frac H2\right\}\right|\\
\le&\left|\left\{(p,q): q\le\sqrt X,\ p\mid Q\right\}\right|+\left|\left\{(p,q): p\le\sqrt X,\ q\mid Q\right\}\right|\\
\le&2|\mathcal F|\pi(\sqrt X)=o\left(\pi(\sqrt X)^2\right),
\end{align*}
it follows that
\[
\left|\left\{(p,q):p,q\le\sqrt X,\
T(p^2+q^2)<\frac H2\right\}\right|
\le
\left(\frac{10}H+o_{\mathcal F}(1)\right)\pi(\sqrt X)^2,
\]
which proves \eqref{eq:prime-count}. This completes the proof of Lemma \ref{lem:primes}.
\end{proof}

\section{Proofs of the main results}

Throughout this section, let $A=\{a_k:k\ge1\}$ and let
\[
\Sset(A)=\{p^2+q^2+a_k+a_l:
p,q\text{ are primes and }k,l\ge1\}.
\]
Since the sequence is strictly increasing, we have
$A(X)=|\{k:a_k\le X\}|$.

\begin{proof}[Proof of Theorem~\ref{thm:zero}]

Let
\[
\mathcal P=\{q:q\mid a_k\text{ for some }k\ge1\}.
\]
For every integer $k\ge1$, write
\[
a_k=\prod_{q\mid a_k}q^{\nu_q},
\qquad \nu_q\ge1.
\]
Since
\[
\sum_{j=0}^{\infty}\frac1{q^j}
=\left(1-\frac1q\right)^{-1},
\]
it follows that
\[
\begin{aligned}
\log\left(\sum_{d\mid a_k}\frac1d\right)
&=\log\left(\prod_{q\mid a_k} \left(1+\frac1q+\frac1{q^2}+\cdots+\frac1{q^{\nu_q}}\right)\right)\\
&\le\log\left(\prod_{q\mid a_k}\left(1-\frac1q\right)^{-1}\right)\\
&=\sum_{q\mid a_k}\log\left(1+\frac1{q-1}\right)\\
&\le\sum_{q\mid a_k}\frac1{q-1}
\le 2\sum_{q\mid a_k}\frac1q.
\end{aligned}
\]
Since
\[
\sup_{k\ge1}\sum_{d\mid a_k}\frac1d=\infty,
\]
it follows that
\begin{equation}\label{eq:divergence}
\sum_{q\in\mathcal P}\frac1q=\infty.
\end{equation}

Fix a nonempty finite set $\mathcal F\subseteq\mathcal P\setminus\{2\}$,
and let $H,Q,T$ be as in Lemma~\ref{lem:local}. Since $\mathcal F$ is a finite set and $a_k\mid a_{k+1}$ for every integer $k\ge1$, it follows that there exists $K\ge1$ such that $Q\mid a_K$, and so
\begin{equation}\label{eq:tail}
Q\mid a_k\quad(k\ge K),\qquad
T(p^2+q^2+a_k+a_l)=T(p^2+q^2)
\quad(k,l\ge K).
\end{equation}

Since the positive sequence $\{a_k\}_{k\ge1}$ is strictly increasing and $a_k\mid a_{k+1}$ for every integer $k\ge1$, it follows that
\[
a_k\ge2a_{k-1}\ge\cdots\ge2^{k-1}a_1\ge2^{k-1}.
\]
Therefore
\begin{equation}\label{eq:sparse}
A(X)\le1+\frac{\log X}{\log2}.
\end{equation}
If $n\le X$ and $n=p^2+q^2+a_k+a_l$, then
$p,q\le\sqrt X$ and $a_k,a_l\le X$. The number of integers of the form $n=p^2+q^2+a_k+a_l$ with $n\le X$ and either $k<K$ or
$l<K$ is at most
\begin{equation}\label{eq:prefix}
2(K-1)A(X)\pi(\sqrt X)^2
=O_{\mathcal F}\left(\frac X{\log X}\right)=o_{\mathcal F}(X).
\end{equation}

For any integer of the form $n=p^2+q^2+a_k+a_l$ with $n\le X$ and $k,l\ge K$, we consider two cases.

{\bf Case 1.~}$T(n)\ge H/2$. By \eqref{eq:ordinary-count}, the number of such integers is at most
\begin{equation}\label{eq:ge}
\left|\left\{1\le n\le X:T(n)\ge\frac H2\right\}\right|
\le\frac{4X}H+O_{\mathcal F}(1).
\end{equation}

{\bf Case 2.~}$T(n)<H/2$. By \eqref{eq:tail} and \eqref{eq:prime-count}, the number of such integers is at most
\begin{equation}\label{eq:le}
A(X)^2\left|\left\{(p,q):p,q\le\sqrt X,\
T(p^2+q^2)<\frac H2\right\}\right|
\le
A(X)^2\left(\frac{10}H+o_{\mathcal F}(1)\right)\pi\left(\sqrt X\right)^2.
\end{equation}
By \eqref{eq:sparse} and the prime number theorem,
\[
\begin{aligned}
\frac{A(X)^2\pi(\sqrt X)^2}{X}
&\le
\left(1+\frac{\log X}{\log 2}\right)^2
\frac{4+o(1)}{(\log X)^2}\\
&=(4+o(1))
\left(\frac1{\log X}+\frac1{\log 2}\right)^2\\
&=\frac4{(\log 2)^2}+o(1).
\end{aligned}
\]
It follows from \eqref{eq:prefix}, \eqref{eq:ge} and \eqref{eq:le} that
\[
\frac{\Sset(A)(X)}{X}
\le\frac1H\left(4+\frac{40}{(\log2)^2}\right)+o_{\mathcal F}(1).
\]
Then
\begin{equation}\label{eq:upper-density}
\overline d(\Sset(A))
\le\frac1H\left(4+\frac{40}{(\log2)^2}\right).
\end{equation}
By \eqref{eq:divergence}, the finite set $\mathcal F$ can be chosen such that $H$ is arbitrarily large. Letting $H\to\infty$ in
\eqref{eq:upper-density}, we have
\[
\overline d(\Sset(A))=0.
\]
This completes the proof of Theorem~\ref{thm:zero}.
\end{proof}

\begin{proof}[Proof of Corollary~\ref{cor:criterion}]
The sufficiency follows from Theorem \ref{thm:daipan}.

Suppose that $\underline d(\Sset(A))>0$.
By Theorem~\ref{thm:zero}, condition \eqref{eq:divisors} holds.
For completeness, we give the short counting argument for
\eqref{eq:growth}.

If \eqref{eq:growth} does not hold, there are indices $j\to\infty$
such that $\log a_j/j\to\infty$. Let $X_j=a_j-1$.
Since $A(X_j)=j-1$ and $\log X_j\sim\log a_j$, it follows that
\[
\frac{\Sset(A)(X_j)}{X_j}
\le\frac{(j-1)^2\pi(\sqrt{X_j})^2}{X_j}
=O\!\left(\left(\frac j{\log X_j}\right)^2\right)=o(1),
\]
a contradiction with $\underline d(\Sset(A))>0$.
Hence \eqref{eq:growth} holds. This completes the proof of Corollary~\ref{cor:criterion}.
\end{proof}

\noindent\textbf{Acknowledgments.} This work was supported by the National Natural Science Foundation of China (Grant No. 12371005).

\end{document}